\documentclass[a4paper,11pt ]{article}
\usepackage[top=3.5cm,bottom=3.5cm,right=3cm,left=3cm]{geometry}
\usepackage[T1]{fontenc}
\usepackage[utf8]{inputenc}

\usepackage{amsfonts}
\usepackage{amsmath}
\usepackage{amssymb}
\usepackage{amscd}
\usepackage{amsthm}
\usepackage{mathrsfs}

\usepackage[all]{xy}

\newtheorem{theorem}{Theorem}[subsection]
\newtheorem{lemma}{Lemma}[subsection]
\newtheorem{proposition}{Proposition}[subsection]

\newtheorem{defi}{Definition}

\begin{document}

\date{}
\title{\large{FROBENIUS AND NON LOGARITHMIC RAMIFICATION }}

\author{{\bf Stéphanie Reglade}}

\maketitle

{\footnotesize \noindent\textbf{Abstract}: \textit{A $\ell$-extension is said logarithmically unramified if it is locally cyclotomic.
The purpose of this article is to explain the construction of the logarithmic Frobenius, which plays the role usually played by the classical Frobenius, but in the context of the logarithmic ramification.The interesting point is that usual and logarithmic Frobenius coincide when usual and logarithmic ramification are the same.}}
\smallskip

\medskip

\noindent \textbf{Key words}: \textbf{class field theory, $\ell$-adic class field theory.}
\medskip

\noindent \textbf{AMS Classification: 11R37}
\medskip

\tableofcontents
\bigskip

\noindent\textbf{\large Introduction}: \medskip

\noindent The notion of logarithmic ramification was developped by Jaulent in  {~\cite{2}}.
The starting point of this article is the following fact: let $L/K$ be an extension of number fields, let
$\mathfrak{p}$ be a prime of $K$, $\mathfrak{p}$ is logarithmically unramified in $L$ if
 $L_{\mathfrak{P}} \subseteq \hat{K_{\mathfrak{p}}^{c}}$, where $\hat{K_{\mathfrak{p}}^{c}}$ is the $\hat{\mathbb{Z}}$-cyclotomic extension of $K_{\mathfrak{p}}$.
\noindent Consequently the decomposition sub-group is cyclic. Thus we may naturally wonder if there exists, in
this decomposition sub-group, an element which is going to play the role usually played by the classical Frobenius ?

\medskip

\noindent Following Neukich's abstract theory ~\cite[chap.II]{1}, we first build the logarithmic local symbol. Neukirch's context starts with an abstract Galois theory and an abstract profinite group $G$. The key points of this formal theory are two fondamental morphims (the degree map and the valuation) and the class field axiom (cohomological condition on the  $G$-module $A$ we work with).

 \noindent In this article, the local object we study is the $\ell$-adification of the multiplicative group of a local field defined by Jaulent in \cite{2}. It is endowed with the logarithmic valuation (we recall its construction). We first define the degree map. We then prove that Neukich's abstract theory applies. This allows us to define the logarithmic local symbol. In particular we get an explicit expression for the logarithmic local symbol on the $\ell$-adification of $\mathbb{Q}_{p}$: 
\smallskip

\noindent \textbf{Proposition.}
\noindent \textit {Let $\zeta$ be a root of unity of order a $\ell$-th power, and $a \in \mathcal{R}_{\mathbb{Q}_{p}}==\mathbb{Z}_{\ell} \otimes_{\mathbb{Z}} \mathbb{Q}^{\times}_{p}$. The logarithmic local symbol is:
$$(a,(\mathbb{Q}_p (\zeta)/\mathbb{Q}_p)_{\ell})(\zeta)= \zeta^{n_{p}}$$ with
\begin{equation*}
  n_p=
     \begin{cases}
        p^{v_p(a)} & \text{ for $p\ne \ell$ and $ p\ne\infty $} \\
       (1+\ell)^{-\tilde v_{\ell}(a)} & \text{ for $p=\ell$}\\
        sgn(a) & \text{ for $p = \infty $ }
     \end{cases}
\end{equation*} 
\noindent where $ (\mathbb{Q}(\zeta)/\mathbb{Q})_\ell$ denotes the projection on the $\ell$-Sylow sub-group of  $\textrm{Gal}(\mathbb{Q}(\zeta)/\mathbb{Q}).$ }

\smallskip

\noindent Thus we study the global case. Finally, after choosing the logarithmic uniformizer  $\tilde \pi_{\mathfrak{p}} $, we obtain the explicit construction of the logarithmic Frobenius:
\medskip

\noindent \textbf{Definition.}
\noindent \textit{Let $L/K$ be an abelian $\ell$-extension of number fields. Let $\mathfrak{p}$ be a prime of $K$ logarithmically unramified in $L$. The logarithmic Frobenius attached to $\mathfrak{p}$ is:
$$ (\widetilde{\frac{L/K}{\mathfrak{p}}})=([\tilde{\pi_{\mathfrak{p}}}], L/K)$$
\noindent where $[\tilde \pi_{\mathfrak{p}}]$ is the image of the uniformizing element $\tilde \pi_{\mathfrak{p}}$, through the logarithmic global symbol defined on the $\ell$-adic idele group $\mathcal{J}_K=\prod \mathcal{R}_{K_{\mathfrak{p}}}$. }
\smallskip

\noindent We are now able to extend this map by multiplicativity. By this way,  we obtain the logarithmic Artin map:
\smallskip

\noindent \textbf{Definition.}
\noindent \textit{Let $L/K$ be a finite and abelian $\ell$-extension and $\mathfrak{p}$ a prime of $K$ logarithmically unramified in $L$. Let $D\ell_{K}$ be the group of logarithmic divisors of $K$, $\tilde{\mathfrak{f}}_{L/K}$ the logarithmic global conductor of $L/K$ and $D\ell_{K}^{\tilde{\mathfrak{f}}_{L/K}}$ the sub-module of logarithmic divisors prime to the conductor $\tilde{\mathfrak{f}}_{L/K}$. We define the logarithmic Artin map on $D\ell_{K}^{\tilde{\mathfrak{f}}_{L/K}}$ as follows: }

\begin{center}

$$\widetilde{(\frac{L/K}{.})} : \mathfrak{p} \in D\ell_{K}^{\tilde{\mathfrak{f}}_{L/K}}  \mapsto  (\widetilde{\frac{L/K}{\mathfrak{p}}}) \in \mathrm{Gal}(L/K) $$
\end{center}

\smallskip

\noindent We study the properties of this map and give an expression of its kernel.
We then focus on the quadratic case and generalize it to the case of a $\ell$-extension.
\bigskip

\newpage
\noindent \textbf{Notations}

\noindent{Let $\ell$ be a fixed prime number. Let's introduce the notations.}
\bigskip

\noindent{\textit{For a local field $K_{\mathfrak{p}}$ with maximal ideal $\mathfrak{p}$ and uniformizer $\pi_{\mathfrak{p}}$, we let:}

$\mathcal{R}_{K_{\mathfrak{p}} }=\varprojlim_{k}  K_{\mathfrak{p}}^{\times} \diagup {  K_{\mathfrak{p}}^{\times  \ell^{k}}}$:  
the $\ell$-adification of the multiplicative group of a local field

$\mathcal{U}_{K_{\mathfrak{p}} }=\varprojlim_{k} {U}_{\mathfrak{p}}  \diagup  U_{\mathfrak{p}}^{\ell^k}$: the $\ell$-adification  of the group of units $U_{\mathfrak{p}} $  of $K_{\mathfrak{p}}$

$U_{\mathfrak{p}}^{1} $: the group of principal units of  $K_{\mathfrak{p}}$

$\mu_{\mathfrak{p}}^{0}$: the subgroup of $ U_{\mathfrak{p}}$, whose order is  finite and prime to $\mathfrak{p}$

$\mu_{\mathfrak{p}}$: the $\ell $- Sylow subgroup of $\mu_{\mathfrak{p}}^{0}$
\smallskip

\noindent{\textit{For a number field $K$ we define: }}

$\mathcal{R}_{K}=\mathbb{Z}_{\ell} \otimes_{\mathbb{Z}} K^{\times}$ : the
$\ell$-adic group of principal ideles

 $\mathcal{J}_{K} =\prod_{ \mathfrak{p} \in Pl_{K} } ^{res} \mathcal{R}_{K_{\mathfrak{p}}}$ : the $\ell$-adic idele group

$\mathcal{U}_{K}=\prod_{ \mathfrak{p} \in Pl_{K} } \mathcal{U}_{K_{\mathfrak{p}}}$ 
: the subgroup of units

$\mathcal{C}_{K}= \mathcal{J}_{K} /  \mathcal{R}_{K}$ : the $\ell$-adic idele class group 
\smallskip

\noindent{\textit{In the logarithmic context, we denote:}}

$\hat{\mathbb{Q}_{\mathfrak{p}}^{c}}$ : the cyclotomic $\widehat{\mathbb{Z}}$-extension of $\mathbb{Q}_{p}$

$\mathbb{Q}_{p}^{c}$ : the cyclotomic $\mathbb{Z}_{\ell}$-extension of $\mathbb{Q}_{p}$

$\tilde v_{\mathfrak{p}}$ : the logarithmic valuation attached to $\mathfrak{p}$ sur $\mathcal{R}_{K_{\mathfrak{p}}}$

$\widetilde{\mathcal{U}}_{ K_{\mathfrak{p}}}=\textrm{Ker}(\tilde{v_{\mathfrak{p}}})$ : the sub-group of local logarithmic units

$\widetilde{\mathcal{U}}_{K}=\prod_{ \mathfrak{p} \in Pl_{K} } \widetilde{\mathcal{U}}_{K_{\mathfrak{p}}}$ 
: the sub-group of logarithmic units

 $\tilde e_{\mathfrak{p}} = [ K_{\mathfrak{p}} : \hat{\mathbb{Q}_{\mathfrak{p}}^{c}} \cap K_{\mathfrak{p}} ]$ :
the logarithmic absolute ramification index of $\mathfrak{p}$

$ \tilde f_{\mathfrak{p}}=[ \hat{\mathbb{Q}_{\mathfrak{p}}^{c}} \cap K_{\mathfrak{p}} : \mathbb{Q}_{\mathfrak{p}}]$ :
the logarithmic absolute inertia degree of $\mathfrak{p}$

$ \tilde e_{ L_{\mathfrak{P}}/K_{\mathfrak{p}}} = [ L_{\mathfrak{P}} :
\hat{K_{\mathfrak{p}}^{c}} \cap L_{\mathfrak{P}} ]$ : the relative logarithmic ramification index of $\mathfrak{p}$

$ \tilde f_{ L_{\mathfrak{P}}/K_{\mathfrak{p}}}=[ \widehat{K_{\mathfrak{p}}^{c}} \cap L_{\mathfrak{P}} : K_{\mathfrak{p}}]$ :
the logarithmic relative inertia degree of $\mathfrak{p}$

 $\mathcal{J}_{K}^{(\mathfrak{m})}=\prod_{\mathfrak{p} \nmid \mathfrak{m}} \mathcal{R}_{K_{\mathfrak{p}}} \prod_{\mathfrak{p} \vert \mathfrak{m} } \widetilde{\mathcal{U}}_{K_{\mathfrak{p}}}^{v_{\mathfrak{p}}(\mathfrak{m})}$

$\mathcal{J}_{K}^{\mathfrak{m}}=\prod_{\mathfrak{p} \nmid \mathfrak{m} } \mathcal{R}_{K_{\mathfrak{p}}} \prod_{\mathfrak{p} \vert \mathfrak{m} }
\widetilde{\mathcal{U}}_{K_{\mathfrak{p}}} $

$\widetilde{\mathcal{U}}_{K}^{(\mathfrak{m})}=\prod_{\mathfrak{p} \nmid \mathfrak{m} } \mathcal{U}_{K_{\mathfrak{p}}} \prod_{\mathfrak{p} \vert \mathfrak{m} } \widetilde{\mathcal{U}}_{K_{\mathfrak{p}}}^{v_{\mathfrak{p}}(\mathfrak{m})} $

$\mathcal{R}_{K}^{(\mathfrak{m})}=\mathcal{R}_{K} \cap \mathcal{J}_{K}^{(\mathfrak{m})}$

\smallskip

\centerline{
$\tilde{div}: \alpha=(\alpha_{\mathfrak{p}}) \in \mathcal{J}_{K}^{\mathfrak{m}} \longmapsto \tilde{div}(\alpha)=\prod \mathfrak{p}^{\tilde{v}_{\mathfrak{p}}(\alpha_{\mathfrak{p}})} \in  D\ell_{K}^{\mathfrak{m}} $}
\smallskip

$ \mathcal{D}\ell_{K}^{\mathfrak{m}}=\tilde{div}(\mathcal{J}_{K}^{\mathfrak{m}})$ : logarithmic divisors prime to $m$

$ \mathcal{P}\ell_{K}^{(\mathfrak{m})}=\tilde{div}(\mathcal{R}_{K}^{(\mathfrak{m})})$ : principal logarithmic divisors attached to  $m$

$\widetilde{\mathfrak{f}}_{L/K}$: the logarithmic global conductor $L/K$

$\widetilde{\mathfrak{f}}_{\mathfrak{p}} $: the logarithmic local conductor attached to $\mathfrak{p}$

$(\widetilde{\frac{L/K}{\mathfrak{p}} })$: the logarithmic Frobenius attached to $\mathfrak{p}$

$\mathcal{A}\ell_{L/K}$ the logarithmic Artin group of $L/K$

\newpage
\section{Recall: Logarithmic ramification }

\noindent We first recall the notion of logarithmic ramification developped by Jaulent in \cite{2}.

\noindent{Let $K, L$ be number fields, $p$ a prime number, $\mathfrak{p}$  a prime of $K$ above $p$ and $\mathfrak{P}$ a prime of $L$ lying above $\mathfrak{p}$. Let's denote $\widehat{\mathbb{Q}_{p}^{c}}$ the $\widehat{\mathbb{Z}}$-cyclotomic extension of
 $\mathbb{Q}_{p}$, i.e. the compositum of all $\mathbb{Z}_{q}$-cyclotomic extensions of $\mathbb{Q}_{p}$ for all primes $q$.}

\begin{defi}{Absolute and relative indexes :} ~\cite[definition 1.3]{2} 

\begin{itemize}

\item[i)] the absolute and relative logarithmic ramification index of $\mathfrak{p}$ are respectively: 
 $$\tilde e_{\mathfrak{p}} = [ K_{\mathfrak{p}} : \hat{\mathbb{Q}_{\mathfrak{p}}^{c}} \cap K_{\mathfrak{p}} ] \qquad \tilde e_{ L_{\mathfrak{P}}/K_{\mathfrak{p}}} = [ L_{\mathfrak{P}} :
\hat{K_{\mathfrak{p}}^{c}} \cap L_{\mathfrak{P}} ]$$

\item[ii)] the absolute and relative logarithmic inertia degree of $\mathfrak{p}$ are respectively: 
$$ \tilde f_{\mathfrak{p}}=[ \hat{\mathbb{Q}_{\mathfrak{p}}^{c}} \cap K_{\mathfrak{p}} : \mathbb{Q}_{\mathfrak{p}}] \qquad
\tilde f_{ L_{\mathfrak{P}}/K_{\mathfrak{p}}}=[ \widehat{K_{\mathfrak{p}}^{c}} \cap L_{\mathfrak{P}} : K_{\mathfrak{p}}]$$

\item[iii)]$K/\mathbb{Q}$ is said logarithmically unramified at $\mathfrak{p}$ if $\tilde e_{\mathfrak{p}}=1$,
which means $K_{\mathfrak{p}} \subseteq \widehat{\mathbb{Q}_{p}^{c}}$. 

\item[iv)]$L/K$ is said logarithmically unramified at $\mathfrak{p}$  if $\tilde e_{ L_{\mathfrak{P}}/K_{\mathfrak{p}}}=1$, which implies $ L_{\mathfrak{P}} \subseteq \widehat{K_{\mathfrak{p}}^{c}}$.

\item[v)] These indexes satisfy the relations: $ \tilde e_{\mathfrak{P}}  = \tilde e_{ L_{\mathfrak{P}}/K_{\mathfrak{p}}} .\tilde e_{\mathfrak{p}} $
\; \textrm{et} \; $ \tilde f_{\mathfrak{P}}  = \tilde f_{ L_{\mathfrak{P}}/K_{\mathfrak{p}}} .\tilde f_{\mathfrak{p}} $

\end{itemize}

\end{defi}

\noindent{According to the diagramm:~\cite[1.1.3]{7} }

\[
 \xymatrix{
     &    & L_{\mathfrak{P}} \ar@{-}[d]^{\tilde e_{ L_{\mathfrak{P}}/K_{\mathfrak{p}}}}  \ar@{-}[ld]_{\tilde f_{\mathfrak{P}}} \\
      & L_{\mathfrak{P}} \cap \widehat{\mathbb{Q}_{p}^{c}}  \ar@{-}[r]^{\tilde e_{\mathfrak{p}}}  \ar@{-}[d] \ar@{-}[ld]_{\tilde e_{\mathfrak{P}}}  & \widehat{K_{\mathfrak{p}}^{c}} \cap L_{\mathfrak{P}} \ar@{-}[d]^{\tilde f_{ L_{\mathfrak{P}}/K_{\mathfrak{p}}}}  \\
\mathbb{Q}_{p} \ar@{-}[r]^{\tilde f_{\mathfrak{p}}} & K_{\mathfrak{p}} \cap \widehat{\mathbb{Q}_{p}^{c}} \ar@{-}[r]^{\tilde e_{\mathfrak{p}}}  & K_{\mathfrak{p}} \\
    }
\]

\begin{proposition}
 ~\cite[theorem 1.4]{2}
With the pevious notations, classical and logarithmic indexes are linked by this formula:
\begin{center}
$\tilde e_{\mathfrak{p}}.\tilde f_{\mathfrak{p}} =e_{\mathfrak{p}}.f_{\mathfrak{p}}=[K_{\mathfrak{p}}:\mathbb{Q}_{p}]$
\end{center}

\end{proposition}

\noindent{\textbf{Fundamental remark}:  ~\cite[p. 4]{2}}
Assume $K/\mathbb{Q}$ is a finite $\ell$-extension such that $[K:\mathbb{Q}]=\ell^{n}$ with $n \ge 1$.
Then $\widehat{\mathbb{Q}_{p}^{c}} / \mathbb{Q}_{p}^{c}$ only contains sub-extensions of order prime to $\ell$.

\noindent In particular the degree of $[\widehat{\mathbb{Q}_{p}^{c}} \cap K_{\mathfrak{p}} : \mathbb{Q}_{p}^{c} \cap K_{\mathfrak{p}}]$ is prime to $\ell$  and as it also divides $\ell^{n}$, we deduce $[\widehat{\mathbb{Q}_{p}^{c}} \cap K_{\mathfrak{p}} : \mathbb{Q}_{p}^{c} \cap K_{\mathfrak{p}}]=1$. The equality of fields $\widehat{\mathbb{Q}_{p}^{c}} \cap K_{\mathfrak{p}} = \mathbb{Q}_{p}^{c} \cap K_{\mathfrak{p}}$ implies this equivalence:
$$\tilde e_{\mathfrak{p}}=1  \Leftrightarrow K_{\mathfrak{p}} \subseteq \mathbb{Q}_{p}^{c} \qquad \textrm{and} \qquad
\tilde f_{\mathfrak{p}}=1 \Leftrightarrow \mathbb{Q}_{p}^{c} \cap K_{\mathfrak{p}}=\mathbb{Q}_{p}^{c}.$$
As $\mathbb{Q}_{p}^{c}/\mathbb{Q}_{p}$ is a  Galois extension, the previous condition means that the extensions  $K_{\mathfrak{p}}$ et $\mathbb{Q}_{p}^{c}$  are linearly separated on $\mathbb{Q}_{p}$. 

\noindent As we work with $\ell$-extensions, we may replace in the previous definitions $\widehat{\mathbb{Q}_{p}^{c}}$ by $\mathbb{Q}_{p}^{c}$.

\section{The local case}
\subsection{The degree map}

\noindent Let $K_{\mathfrak{p}}^{ab}$ be the maximal abelian pro-$\ell$-extension of $K_{\mathfrak{p}}$. On the Galois group  $\textrm{Gal}(K_{\mathfrak{p}}^{ab}(\zeta_{\ell^{ \infty}})/K_{\mathfrak{p}})$, the Teichmuller's character $\omega$ is defined as the character of the action on roots of unity. It  is defined for pro-$\ell$-extensions as the restriction to the pro-$\ell$-part of the whole character.

\noindent We thus define the degree map: 

$$\begin{array}{ccccc}
\textrm{deg}  & : &  G=\textrm{Gal}(K_{\mathfrak{p}}^{ab}/K_{\mathfrak{p}}) & \to &  \mathbb{Z}_{\ell}  \\
& & \phi& \mapsto & \omega(\phi)\\
\end{array}$$

\noindent where $K_{\mathfrak{p}}^{ab}$ is the maximal abelian pro-$\ell$-extension of $K_{\mathfrak{p}}$.
\smallskip

\noindent  We then follow Neukich's abstract construction. If $L_{\mathfrak{P}}$ is a finite $\ell$-extension of
$K_{\mathfrak{p}}$, the logarithmic ramification index and the logarirhmic inertia degree appear naturally:  
\begin{center}
$\tilde f_{L_{\mathfrak{P}}/K_{\mathfrak{p}}}=[ L_{\mathfrak{P}} \cap K_{\mathfrak{p}}^{c} : K_{\mathfrak{p}}]
\qquad \tilde e_{L_{\mathfrak{P}}/K_{\mathfrak{p}}}=[L_{\mathfrak{P}} : L_{\mathfrak{P}} \cap K_{\mathfrak{p}}^{c}] $
\end{center}

\noindent{Those definitions coincide with those given by  Jaulent~\cite{2}  . }

\noindent{\textbf{Remark: }}
\noindent{We replace here in the definitions  $\hat{K_{\mathfrak{p}}^{c}}$ by
 $K_{\mathfrak{p}}^{c}$ because we work with $\ell$-extensions. }

\subsection{The $G$-module}

\noindent We introduce the $\ell$-adification of the multiplicative group of a local field: $\mathcal{R}_{L_{\mathfrak{P}}} =\varprojlim_{k}  L_{\mathfrak{P}}^{\times}   \diagup {  L_{\mathfrak{P}}^{\times  \ell^{k}}}$  defined by Jaulent 
 ~\cite [def.1.2]{3}. The $G$-module we study is the same as before ~\cite [\textsection 2.2]{5}: 
$A=\lim\limits \mathcal{R}_{L_{\mathfrak{P}}}$
\noindent, where $L_{\mathfrak{P}}$ runs through all finite sub-extensions of $K_{\mathfrak{p}}$. It can be canonically identified to: 
$ A= \bigcup_{[L_{\mathfrak{P}}:K_{\mathfrak{p}}]< \infty}  \mathcal{R}_{L_{\mathfrak{P}}}.$
 If $L_{\mathfrak{P}}$ is a finite extension of $K_{\mathfrak{p}}$, $A_{L_{\mathfrak{P}}}=\mathcal{R}_{L_{\mathfrak{P}}}$
is the $Gal(L_{\mathfrak{P}}/K_{\mathfrak{p}})$-module we are going to work with.

\subsection{Logarithmic valuation and $\ell$-adic degree}

\begin{defi}
Let $K$ be a finite extension of $\mathbb{Q}$, $p$ a prime number. Let's denote $\widehat{\mathbb{Q}_{p}^{c}}$ the $\widehat{\mathbb{Z}}$-cyclotomic extension of $\mathbb{Q}_{p}$ and let $\mathfrak{p}$ be a prime of $K$ above $p$.

\begin{itemize}

\item[i)] The $\ell$-adic degree of $p$ is:
$ deg_ {\ell}(p)
= \left\{
    \begin{array}{lll}
      Log_{Iw}(p) & \mbox{ if } p \ne  \ell \\
      Log_{Iw}(1+\ell) & \mbox{ if }  p=\ell 
    \end{array}
\right.$

\item[ii)] The $\ell$-adic degree of  $\mathfrak{p}$ is:
$deg(\mathfrak{p})=\tilde{f_{\mathfrak{p}}}\cdot deg_{\ell}(p)$

\item[iii)] Let $v_{\mathfrak{p}}$ be the usual normalized valuation on $K_{\mathfrak{p}}$, 
the absolute principal  $\ell$-adic valuation, defined on $K_{\mathfrak{p}}^{\times}$ is: 
\begin{center}
if $\mathfrak{p} \nmid \ell$ then $ |x_{\mathfrak{p}}|=<N \mathfrak{p} ^{- v_{\mathfrak{p}}(x)}>$
\end{center}
\begin{center}
\quad \quad \qquad if $\mathfrak{p} \vert \ell $ then $ |x_{\mathfrak{p}}|=<N_{ K_{\mathfrak{p}}/ \mathbb{Q}_{\ell}} (x)   N \mathfrak{p} ^{- v_{\mathfrak{p}}(x)}>$
\end{center}

\noindent where $N \mathfrak{p}$ is the absolute norm of $\mathfrak{p}$ and $ u \longrightarrow <u>$   is the canonical surjection from ${\mathbb{Z}_{\ell}}^{\times}$ to the group of principal units.

\item[iv])the logarithmic valuation attached to $\mathfrak{p}$ is: 
$ \tilde v_{\mathfrak{p}}(x)= -Log_{Iw}(N_{K_{\mathfrak{p}}/\mathbb{Q}_{p}}(x))/deg_{\ell}(\mathfrak{p}) $
, defined on $\mathcal{R}_{K_{\mathfrak{p}}}$  valued to $\mathbb{Z}_{\ell}$,~\cite{1}[Prop.1.2]

\end{itemize}
\end{defi}

\noindent{\textbf{Remark: }}
\noindent{This definition of the logarithmic valuation is here different from the one given by Jaulent ~\cite[définition 1.1]{2}. Indeed the $\ell$-adic degree of $\ell$, $\textrm{deg}_{\ell}(\ell)$, is equal to $\ell$ initially and to  $ Log_{Iw}(1+\ell)$ in our case. This is motivated by the fact that we want an explicit expression for the logarithmic uniformizers.

\begin{proposition}{~\cite[proposition 1.2]{2}}

\noindent{Let $p$ be a prime number, $\mathbb{Q}_{\mathfrak{p}}^{c}$ the $\hat{\mathbb{Z}}$-cyclotomic extension}
of $\mathbb{Q}_{\mathfrak{p}}$, $K_{\mathfrak{p}}/\mathbb{Q}_{\mathfrak{p}}$ a finite extension of degree $d_{\mathfrak{p}}$ and $|.|_{\mathfrak{p}}$  the principal absolute $\ell$-adic valuation on $K_{\mathfrak{p}}^{\times}$. Then:

\begin{itemize}

\item[i)] for all $x \in K_{\mathfrak{p}}^{\times}$ the expression $h_{p}(x)= - Log(|x|_{\mathfrak{p}}/ d_{\mathfrak{p}}.deg_{\ell}(p)$
does not depend on the choice of the extension of $\mathbb{Q}_{\mathfrak{p}}$ containing $x$ ;

\item[ii)] the restriction $h_{\mathfrak{p}}$ of $h_p$ to the multiplicative group of $K_{\mathfrak{p}}^{\times}$, yields a $\mathbb{Z}_{\ell}$-morphism from $\mathcal{R}_{K{\mathfrak{p}}}$ to $\mathbb{Q}_{\ell}$, whose kernel is $\tilde{\mathcal{U}}_{K_{\mathfrak{p}}}$
and whose image is the $\mathbb{Z}_{\ell}$-lattice $1/\tilde e_{\mathfrak{p}} \cdot \mathbb{Z}_{\ell}$

\item[iii)] the logarithmic valuation satisfies: $\tilde v_{\mathfrak{p}}= \tilde e_{\mathfrak{p}} \cdot h_{\mathfrak{p}}$.

\item[iv)] let $\mathfrak{p}$ be a prime of $K$ which is not above $\ell$, classical and logarithmic valuations are proportional:
$$\tilde v_{\mathfrak{p}} = \frac{f_{\mathfrak{p}}}{\tilde f_{\mathfrak{p}}} \cdot v_{\mathfrak{p}} = \frac{\tilde e_{\mathfrak{p}}}{e_{\mathfrak{p}}} \cdot v_{\mathfrak{p}} $$

\end{itemize}
\end{proposition}

\begin{proposition}

\noindent {The logarithmic valuation $\tilde v_{\mathfrak{p}}$ checks two properties:}

\begin{itemize}

\item[i)]  $\tilde v_{\mathfrak{p}}(\mathcal{R}_{K_{\mathfrak{p}}})=Z$ with $ \mathbb{Z} \subset  Z $ et $ Z/n.Z \simeq  \mathbb{Z}/n.\mathbb{Z} $ for all $n$ ;

\item[ii)] $\tilde v_{\mathfrak{p}}(N_{L_{\mathfrak{P}}/K_{\mathfrak{p}}}   \mathcal{R}_{L_{\mathfrak{P}}})=\tilde f_{L_{\mathfrak{P}}/K_{\mathfrak{p}}} \; Z$ for all finite extension $L_{\mathfrak{P}}$ of $K_{\mathfrak{p}}$ .
\end{itemize}

\noindent Thus the logarithmic valuation is henselian with respect to the degree map, according to Neukirch's definition.

\end{proposition}

\begin{proof}

\noindent{For the first criterium (i) we use the definition of the logarithmic valuation :}

 $\tilde v_{\mathfrak{p}}(\mathcal{R}_{K_{\mathfrak{p}}})= \mathbb{Z}_{\ell}$ ;
ainsi $Z= \mathbb{Z}_{\ell}$ et $ Z/n.Z \simeq  \mathbb{Z}/n.\mathbb{Z} $ for all $n$.

\noindent For the second criterium (ii) we use this diagramm~\cite[proposition 1.1]{2} :
$$\begin{CD}
  K_{\mathfrak{ p}}^{\times}@>{ extension}>>  L_{\mathfrak{P}}^{\times}@>{Norm}>> K_{\mathfrak{ p}}^{\times}\\
@VVV @. @VVV{h_{\mathfrak{p}}}\\
\mathbb{Q}_{\ell}  @> >>\mathbb{Q}_{\ell}  @>{[ L_{\mathfrak{P}}:K_{\mathfrak{ p}}}] >> \mathbb{Q}_{\ell}
\end{CD}$$
\noindent{It follows}
$ h_{\mathfrak{p}}( N_{L_{\mathfrak{P}}/K_{\mathfrak{p}}} )=[  L_{\mathfrak{P}}:K_{\mathfrak{p}} ] \; h_{\mathfrak{P}}.$
\noindent{From,}
$\tilde v_{\mathfrak{p}}(N_{L_{\mathfrak{P}}/K_{\mathfrak{p}}}   \mathcal{R}_{L_{\mathfrak{P}}})=
\tilde e_{\mathfrak{p}} \;  h_{\mathfrak{p}}(N_{L_{\mathfrak{P}}/K_{\mathfrak{p}}}   \mathcal{R}_{L_{\mathfrak{P}}})$ 
\noindent{we deduce}
$\tilde v_{\mathfrak{p}}(N_{L_{\mathfrak{P}}/K_{\mathfrak{p}}}   \mathcal{R}_{L_{\mathfrak{P}}})=
  [  L_{\mathfrak{P}}:K_{\mathfrak{p}} ] \; \tilde e_{\mathfrak{p}} \; h_{\mathfrak{P}}(\mathcal{R}_{L_{\mathfrak{P}}})  $
\noindent{thus}
$\tilde v_{\mathfrak{p}}(N_{L_{\mathfrak{P}}/K_{\mathfrak{p}}}   \mathcal{R}_{L_{\mathfrak{P}}})=
[  L_{\mathfrak{P}}:K_{\mathfrak{p}} ] \;  \tilde e_{\mathfrak{p}} /\tilde e_{\mathfrak{P}} \; \mathbb{Z}_{\ell}$
by  ~\cite[proposition 1.2]{2}, then
 $ h_{\mathfrak{P}}(\mathcal{R}_{L_{\mathfrak{P}}})=1/\tilde{e_{\mathfrak{P}}} \; \mathbb{Z}_{\ell}.$
\noindent{According to [proposition 5.1.1],we get:}
$[  L_{\mathfrak{P}}:K_{\mathfrak{p}} ] =\tilde f_{L_{\mathfrak{P}/K_{\mathfrak{p}}}} \; \tilde e_{L_{\mathfrak{P}}/K_{\mathfrak{p}}}$
\noindent{; so}
$[  L_{\mathfrak{P}}:K_{\mathfrak{p}} ] \;  \tilde e_{\mathfrak{p}} / \tilde e_{\mathfrak{P}}=\frac{ \tilde f_{\mathfrak{P}} \; \tilde e_{\mathfrak{P}} \;  \tilde e_{\mathfrak{p}}  }
{\tilde f_{\mathfrak{p}} \; \tilde e_{\mathfrak{p}} \cdot \tilde e_{\mathfrak{P}} }$
\noindent{Thus}
$\tilde v_{\mathfrak{p}} (N_{L_{\mathfrak{P}}/K_{\mathfrak{p}}}   \mathcal{R}_{L_{\mathfrak{P}}})=
\frac{ \tilde f_{\mathfrak{P}}} {\tilde f_{\mathfrak{p}}} \; \mathbb{Z}_{\ell}$
\noindent{and finally}
$\tilde v_{\mathfrak{p}}(N_{L_{\mathfrak{P}}/K_{\mathfrak{p}}}   \mathcal{R}_{L_{\mathfrak{P}}})=
\tilde f_{L_{\mathfrak{P} /K_{\mathfrak{p}}}} \; Z$.

\end{proof}

\subsection{The logarithmic local symbol }

\begin{theorem}
$(\textrm{deg}, \tilde v_{\mathfrak{p}})$ is a class field theory and $\mathcal{R}_{K_{\mathfrak{p}}}$ satisfies the class field axiom ~\cite[theorem 2.5.1]{5}. Thus for all finite and abelian $\ell$-extension $L_{\mathfrak{P}}$ of $K_{\mathfrak{p}}$ (finite extension of $\mathbb{Q}_{p}$) we have this isomorphism :
$$ \textrm{Gal}(L_{\mathfrak{P}}/K_{\mathfrak{p}}) \simeq \mathcal{R}_{K_{\mathfrak{p}}}/N_{L_{\mathfrak{P}}/K_{\mathfrak{p}}}  \mathcal{R}_{L_{\mathfrak{P}}}$$                                                                           
\end{theorem}

\begin{defi}
It allows us to define a surjective homomorphism: the logarithmic local symbol
$$ (\;,L_{\mathfrak{P}}/K_{\mathfrak{p}}) : \mathcal{R}_{K_{\mathfrak{p}}} \longrightarrow \textrm{Gal}(L_{\mathfrak{P}}/K_{\mathfrak{p}})$$
\end{defi}
\smallskip

\noindent{Moreover, for all archimedean place $p$ of $\mathbb{Q}$, we have ~\cite[i), p.~4]{2}}
\begin{equation*}
  \mathcal{R}_{\mathbb{Q}_{p}}=
     \begin{cases}
        \mu_p . p^{\mathbb{Z}_{\ell}} & \text{for $p\ne \ell$ } \\
        (1+ \ell)^{\mathbb{Z}_{\ell}} \cdot \ell^{\mathbb{Z}_{\ell}} & \text{for $ p=\ell$} \\
     \end{cases}
\end{equation*}
Thus we have a decomposition of the shape : $\mathcal{R}_{\mathbb{Q}_{p}} \simeq \widetilde{\mathcal{U}}_{\mathbb{Q}_{p}}. \tilde {\pi_{p}}^{\mathbb{Z}_{\ell}}$.

\begin{proposition}
We have an explicit expression for the logarithmic local symbol.
Let $\zeta$ be a root of unity of $\ell$-th power and $a \in \mathcal{R}_{\mathbb{Q}_{p}}$. The logarithmic local symbol is: 
$$(a,(\mathbb{Q}_p (\zeta)/\mathbb{Q}_p))_{\ell}= \zeta^{n_{p}}$$ with
\begin{equation*}
  n_p=
     \begin{cases}
        p^{v_p(a)} & \text{for $p\ne \ell$ and $ p\ne\infty $} \\
       (1+\ell)^{-\tilde v_{\ell}(a)} & \text{for $p=\ell$}\\
        sgn(a) & \text{for $p = \infty $ }
     \end{cases}
\end{equation*} where $ (\mathbb{Q}(\zeta)/\mathbb{Q})_\ell$ denotes the projection on the $\ell$-Sylow sub-group of $\textrm{Gal}(\mathbb{Q}(\zeta)/\mathbb{Q}).$
\end{proposition}
\smallskip

\noindent \textit{Remark:}
\noindent If $\mathfrak{p}$ is real and $\ell=2$ then we have $\mathcal{R}_{\mathbb{Q}_{\infty}} \simeq \frac {\mathbb{Z}} {2  \mathbb{Z}} $, it is trivial in any other cases, ~\cite[Prop.1.2]{3} :  $sgn(a)$ is $\pm 1$ in the first case and $1$ in the other cases.

\begin{proof}

\noindent{Let $\zeta$ be a $\ell^m$-th root of unity, with $\ell^m \ne 2$. We take $a \in \mathcal{R}_{\mathbb{Q}_{p}}$ and write $ a= u_p \cdot p^{v_p(a)}$ for $p \ne \ell$, where $v_p$ is the usual normalized valuation 
of $\mathbb{Q}_{p}$, which coincide in this particular case with the logarithmic valuation. For $p =\ell$, we write $a= \ell^{v_{\ell}(a)}. (1+\ell)^{\tilde v_{\ell}(a)}$.
For $p \ne \ell$ and $ p \ne \infty$ the extension $\mathbb{Q}_{p}(\zeta)/ \mathbb{Q}_{p}$ is an unramified extension.} The fundamental principle ~\cite[Th. 2.6, p.~25]{1} states that the local symbol associates the Frobenius elements to the uniformizing elements; we have already seen that $(p, (\mathbb{Q}_{p}(\zeta)/ \mathbb{Q}_{p}))_{\ell}$ is the usual Frobenius automorphism $\phi_{p}: \zeta \longrightarrow \zeta^{p}$. Moreover this diagramm is commutative:
 $$\begin{CD}
  K_{\mathfrak{ p}}^{\times}@>{ (\; \cdot \; \dot,\; \textrm{Gal}(L_{\mathfrak{P}}/K_{\mathfrak{p}})) }>>  \textrm{Gal}(L_{\mathfrak{P}}/K_{\mathfrak{p}}) \\
@VVV  @VVV\\
\mathcal{R}_{K_{\mathfrak{p}}}  @>{  (\; \cdot \; \dot ,\; \textrm{Gal}(L_{\mathfrak{P}}/K_{\mathfrak{p}}))_{\ell} } >>\textrm{Gal}(L_{\mathfrak{P}}/K_{\mathfrak{p}})_{\ell}
\end{CD}$$
\noindent where the symbol on the top is the usual local symbol and the symbol on the bottom is the $\ell$-adic local symbol.
That is why we deduce:
$$(a, (\mathbb{Q}_{p}(\zeta)/ \mathbb{Q}_{p})_{\ell}) \zeta=\zeta^{n_p}$$
with
\begin{equation*}
  n_p=
     \begin{cases}
        p^{v_p(a)} & \text{ for $p\ne \ell$ and $ p\ne\infty $} \\
        sgn(a) & \text{ for $p = \infty $ }
     \end{cases}
\end{equation*}
But we want to have $[a, (\mathbb{Q}(\zeta)/\mathbb{Q})_{\ell}]=1$ for $a \in \mathcal{R}_{\mathbb{Q}}$ in order to be able to define the valuation in the logarithmic global context:
$$[a,(\mathbb{Q}(\zeta)/\mathbb{Q})_{\ell}] \zeta = \prod_{\mathfrak{p}}  (a,(\mathbb{Q}_p (\zeta)/\mathbb{Q}_p)_{\ell}) \zeta=\zeta^{\alpha}.$$
Thus, by the product formula, taking $n_{\ell}=(1+\ell)^{-v_{\ell}(a)}$,  we get : $\alpha= \prod_{p} n_p= sgn(a) \cdot \prod_{p \ne \infty } p^{v_p(a)}  \cdot \ell^{-v_{\ell}(a)} \cdot (1+\ell)^{-\tilde v_{\ell}(a)}=a \cdot a^{-1}=1$, as waited.
\end{proof}

\section{The global case}
\subsection{$G$ and the $G$-module}

\noindent Let $G$ be the Galois group of the maximal abelian pro-$\ell$-extension of $\mathbb{Q}$. 
We introduce the $\ell$-adic idele class group for a given number field $K$  refer to  Jaulent
~\cite [definition 1.4]{3}. Like previously~\cite [\textsection 3.4]{5} the $G$-module is the union of all $\ell$-adic idele class groups $\mathcal{C}_K$, where $K$ runs through all finite extensions of $K$ :
$ \bigcup_{[K:\mathbb{Q}] < \infty} \mathcal{C}_K$ and $ \mathcal{C}_L$ is our $\textrm{Gal(L/K)}$-module.

\subsection{The degree map}

\noindent We fix an isomorphism:
$ \textrm{Gal}(\tilde{\mathbb{Q}}/\mathbb{Q}) \simeq \widehat{\mathbb{Z}}.$ This allows to define:
\begin{center}
$\begin{array}{ccccc}
\widetilde{\textrm{deg}}  & : & G= \textrm{Gal}(\mathbb{Q}^{ab}/\mathbb{Q})  & \to &  \hat{\mathbb{Z}}  \\
& & \phi& \mapsto & \phi_{\mid_{\tilde{\mathbb{Q}}}}\\
\end{array}$
\end{center}
where $\mathbb{Q}^{ab}$ is the maximal abelian pro-$\ell$-extension of $\mathbb{Q}$.
Let $K/\mathbb{Q}$ be a finite extension, we define: $ f_K=[K \cap \tilde{\mathbb{Q}} : \mathbb{Q}]$ and we get, by analogy with the abstract case ~\cite [ch.2]{1},  a surjective homomorphism $ \widetilde{\textrm{deg}}_{K}=\frac{1}{f_K} \cdot \widetilde{\textrm{deg}}$ such that 
$\widetilde{ \textrm{deg}}_{K} : \; G_K \longrightarrow \hat{\mathbb{Z}}$ defines the $\hat{\mathbb{Z}}$-extension $\tilde{K}$ of $K$.

\subsection{The valuation}

\noindent
Due to the property, $ \forall a \in \mathcal{R}_K, \, [a, \tilde{K}/K]=1.$ we put this definition:

\begin{defi}
We define the valuation $\widetilde{v}_K : \mathcal{C}_K \longrightarrow \hat{\mathbb{Z}} $ as follows:

$$\begin{CD}
\mathcal{C}_K @>{[\; \cdot \; , \tilde{K}/K]}>> G( \tilde{K}/K) @>{ \widetilde{\textrm{deg}}_{K}}>>\hat{\mathbb{Z}}
\end{CD}$$

\end{defi}

\begin{lemma}
This valuation $\widetilde{v}_K$ is henselian with respect to the degree map $\widetilde{\textrm{deg}}$.
\end{lemma}

\begin{proof}
Arguments are the same as  ~\cite [lemma 3.6.4]{5}  replacing $\mathcal{U}_{K_{\mathfrak{p}}}$ by $\widetilde{\mathcal{U}}_{K_{\mathfrak{p}}}$.
\end{proof}

\begin{theorem}
$(\widetilde{\textrm{deg}}, \widetilde{v}_K)$ is a class field theory and $\mathcal{C}_{K}$ ~\cite[theorem 3.3.1]{4} satisfies the class field axiom. Thus for all finite and abelian $\ell$-extensions  $L$ of $K$, we get an isomorphism:
$$ \textrm{Gal}(L/K) \simeq \mathcal{J}_{K}/N_{L/K} ( \mathcal{J}_{L}) \mathcal{R}_{K}$$                                                                           
\end{theorem}

\begin{defi}
This allows to define a surjective homomorphism, called the global logarithmic symbol:
$$ (\;\cdot\;,L/K) : \mathcal{J}_{K} \longrightarrow \textrm{Gal}(L/K)$$
\end{defi}

\section{The logarithmic Frobenius }

\subsection{Logarithmic uniformizing elements on $\mathbb{Q}_{p}$} \index{uniformisantes logarithmiques}

\noindent \textbf{Logarithmic uniformizing elements on $\mathcal{R}_{\mathbb{Q}_{p}}$ :}

\noindent \textbf{If $p \nmid \ell$} : the classical uniformizing element $p$ is also a uniformizing element for the logrithmic valuation.
\smallskip

\noindent \textbf{If $p=\ell$} :
due to the expression of $\mathcal{R}_{\mathbb{Q}_{\ell}}$, we have $$\mathcal{R}_{\mathbb{Q}_{\ell}}=   \mathcal{U}_{\mathbb{Q}_{\ell}} \widetilde{\mathcal{U}}_{\mathbb{Q}_{\ell}} $$
with $\widetilde{\mathcal{U}}_{\mathbb{Q}_{\ell}}\simeq
\ell^{\mathbb{Z}_{\ell}}$ and $\mathcal{U}_{\mathbb{Q}_{\ell}}\simeq
1+\ell \mathbb{Z}_{\ell}$. 

\noindent We consider a logarithmic uniformizing element $\tilde{\ell}$ such that
$$ \tilde{v}_{\ell}(\tilde{\ell})=1 \; \textrm{et} \; \tilde{\ell} \in \mathcal{U}_{\mathbb{Q}_{\ell}}.$$

\noindent thus we obtain the decomposition:
$$  \mathcal{R}_{\mathbb{Q}_{\ell}}=   \tilde{\ell}^{ \mathbb{Z}_{\ell}}  \ell^{\mathbb{Z}_{\ell}} .$$

\noindent On $\mathbb{Q}_{\ell}$  \textbf{this is the choice of the denominator} in the expression of the logarithmic valuation,  i.e. the $\ell$-adic degree of $\ell$, which \textbf{enforces} $\tilde{\ell}$. 
The condition $Log(\tilde{\ell})=\textrm{deg}_{\ell}(\ell)$, which comes from $\tilde{v}_{\ell}(\tilde{\ell})=1$,
defines $\tilde{\ell}$ up to a logarithmic unit. But we have $\widetilde{\mathcal{U}}_{\mathbb{Q}_{\ell}} \cap \mathcal{U}_{\mathbb{Q}_{\ell}}=1 $, thus $\tilde{\ell}$ is determined by the choice of the denominator.

\noindent For instance, if we choose  the $\ell$-adic degree of $\ell$ as equal to $\textrm{Log}(1+\ell)$, we get: $\tilde{\ell}=1+\ell$.

\subsection{Logarithmic uniformizing elements on $\mathcal{R}_{K_{\mathfrak{p} } }$ }

\noindent \textbf{If} $\mathfrak{p} \nmid \ell$, the classical uniformizing element $\pi_{\mathfrak{p}}$ is also a uniformizer for the logarithmic valuation, in this case both valuations are proportional:
$\pi_{\mathfrak{p}}=\tilde{\pi}_{\mathfrak{p}}$.
\smallskip

\noindent \textbf{If} $\mathfrak{p} \vert \ell$, then we have in this case $\mathcal{R}_{K_{\mathfrak{p}}} \simeq U_{\mathfrak{p}}^{1}
\pi_{\mathfrak{p}}^{\mathbb{Z}_{\ell}}$. By definition a uniformizing element is an element $\tilde{\pi}_{\mathfrak{p}}$
of $\mathcal{R}_{K_{\mathfrak{p}}}$ such that :
$$ Log_{Iw}( N_{K_{\mathfrak{p}}/\mathbb{Q}_{\ell}}( \tilde \pi_{\mathfrak{p}})   )=\tilde{f}_{\mathfrak{p}} \; \textrm{deg}_{\ell}(\ell)=Log_{Iw}(\tilde{\ell}^{\tilde{f}_{\mathfrak{p}} }).$$

\noindent This defines $ \tilde \pi_{\mathfrak{p}}$ up to a logarithmic unit.

\subsection{The logarithmic conductor}

\noindent In usual local class field theory, we have the decreasing filtration of the group of units
$U_{K_{\mathfrak{p}}}$ of $K_{\mathfrak{p}}^{\times}$:
$$U_{K_{\mathfrak{p}}}^{(n)}=1+ \mathfrak{p}^{n}$$
\noindent where $\mathfrak{p}$ is the maximal ideal of the ring of integers.

\noindent $(U_{K_{\mathfrak{p}}}^{(n)})_{n}$ are a system of neighboorhoods of $1$
in $K_{\mathfrak{p}}^{\times}$. Thanks to this filtration, we define the local and the global conductor attached to an extension.

\noindent Let's give a decreasing filtration of the group of logarithmic units $(\widetilde{ \mathcal{U}}_{K_{\mathfrak{p}} }^{n})_{n \in \mathbb{N}}$ with $\widetilde{ \mathcal{U}}_{K_{\mathfrak{p}} }^{0}=\widetilde{ \mathcal{U}}_{K_{\mathfrak{p}} }$.

\smallskip

\begin{defi} 
We are now able to define a logarithmic local and global conductor, as follows:
\begin{itemize}
\item[i)] Let $L_{\mathfrak{P}}/K_{\mathfrak{p}}$ be an abelian $\ell$-extension and $n$ the smallest integer such that 
$\widetilde{ \mathcal{U}}_{K_{\mathfrak{p}}}^{n}  \subseteq N_{L_{\mathfrak{P}}/K_{\mathfrak{p}}}( \mathcal{R}_{L_{\mathfrak{P}}})$ then the ideal:
$$ \tilde {\mathfrak{f}}_{\mathfrak{p}} = \mathfrak{p}^{n}$$
defines the logarithmic local conductor attached to this extension.\index{conducteur logarithmique local}

\item[ii)] Let $L/K$ be a finite and abelian $\ell$-extension, the global logarithlic conductor is:
$$\tilde{\mathfrak{f}}_{L/K}= \prod_{\mathfrak{p}}  \tilde{\mathfrak{f}}_{\mathfrak{p}}  $$
\end{itemize}
\end{defi}

\begin{proposition}

The $\mathfrak{p}$-conductor  $\tilde{\mathfrak{f}}_{\mathfrak{p}}$ is trivial if and only if the extension $L/K$ is logarithmically unramified at $\mathfrak{p}$. The logarithmic global conductor $\tilde{\mathfrak{f}}_{L/K}$ contains all the primes of $K$  which are logarithmically unramified in $L$ and only those.  Besides, if $M$ is between $K$ and $L$ then
$\tilde{\mathfrak{f}}_{M/K}$ divides $ \tilde{\mathfrak{f}}_{L/K}$.
\end{proposition}

\begin{proof}

If $\mathfrak{p}$ is a prime of $K$  logarithmically unramified in $L$, we have $\tilde e_{ L_{\mathfrak{P}}/K_{\mathfrak{p}}}=1$, and due to \cite{1}{Prop.2.2.p.22} :
$$H^{0}(\textrm{Gal}(L_{\mathfrak{P}}/K_{\mathfrak{p}}), \widetilde{ \mathcal{U}}_{L_{\mathfrak{P}}} )=1 ; $$
it follows:
$$  \widetilde{ \mathcal{U}}_{K_{\mathfrak{p}}} = N_{ L_{\mathfrak{P}}/K_{\mathfrak{p}} }  ( \widetilde{ \mathcal{U}}_{L_{\mathfrak{P}}} ), $$
which means
$$   \widetilde{ \mathcal{U}}_{K_{\mathfrak{p}}} \subseteq N_{L_{\mathfrak{P}}/K_{\mathfrak{p}}}  ( \mathcal{R}_{L_{\mathfrak{P}}}).$$
\smallskip

\noindent{Conversely, let's assume that $\mathfrak{f}_{\mathfrak{p}}$ is trivial, we deduce}
$   \widetilde{ \mathcal{U}}_{K_{\mathfrak{p}}}  \subseteq N_{L_{\mathfrak{P}}/K_{\mathfrak{p}}}  ( \mathcal{R}_{L_{\mathfrak{P}}})$. Let $n= [\mathcal{R}_{K_{\mathfrak{p}}} :  N_{L_{\mathfrak{P}}/K_{\mathfrak{p}}}  ( \mathcal{R}_{L_{\mathfrak{P}}})] $ ; from $\tilde \pi_{\mathfrak{p}}^{n} \in  N_{L_{\mathfrak{P}}/K_{\mathfrak{p}}}  ( \mathcal{R}_{L_{\mathfrak{P}}}) $ we get $ (\tilde \pi_{\mathfrak{p}}^{n}) \widetilde{ \mathcal{U}}_{K_{\mathfrak{p}}}  \subseteq N_{L_{\mathfrak{P}}/K_{\mathfrak{p}}}  ( \mathcal{R}_{L_{\mathfrak{P}}})$.
Thus it follows $L_{\mathfrak{P}} \subseteq M$, where $M$ is the class field of $(\tilde \pi_{\mathfrak{p}}^{n})  \widetilde{ \mathcal{U}}_{K_{\mathfrak{p}}}  $ i.e.
$N_{M/K_{\mathfrak{p}}}(\mathcal{R}_M)=(\tilde \pi_{\mathfrak{p}}^{n}) \cdot \widetilde{ \mathcal{U}}_{K_{\mathfrak{p}}}  $. But $ \widetilde{ \mathcal{U}}_{K_{\mathfrak{p}}} \subseteq N_{M/K_{\mathfrak{p}}}(\mathcal{R}_M)$ which means that the $\mathfrak{p}$-conducteur is trivial. Applying the first part, we deduce that $M/K_{\mathfrak{p}}$ is
logarithmically unramified. By $\ell$-adic class field theory, we obtain: $Gal(M/K_{\mathfrak{p}}) \simeq 
(\tilde \pi_{\mathfrak{p}})  \widetilde{ \mathcal{U}}_{K_{\mathfrak{p}}}  / (\tilde \pi_{\mathfrak{p}}^{n})  \widetilde{ \mathcal{U}}_{K_{\mathfrak{p}}} .$ Consequently, $M$ is the $\ell$-extension of degree $n$ logarithmically unramified; and from  $L_{\mathfrak{P}} \subseteq M$, we conclude that $\mathfrak{p}$ is logarithmically unramified in $L$.
\end{proof}

\noindent {Examples}

\begin{itemize}

\item[1)]For $\mathbb{Q}_{\ell}$, we have $ \widetilde{ \mathcal{U}}_{\mathbb{Q}_{\ell}} \simeq \ell^{\mathbb{Z}_{\ell}} \simeq \mathbb{Z}_{\ell}$,
that is why we get a filtration of $ \widetilde{ \mathcal{U}}_{\mathbb{Q}_{\ell}}$ by raising in $ \widetilde{ \mathcal{U}}_{\mathbb{Q}_{\ell}}$ the canonical filtration $\mathbb{Z}_{\ell}$ by $\ell^{n} \mathbb{Z}_{\ell}$, for $n \in \mathbb{N}$.

\item[2)]Given a local field $K_{\mathfrak{p}}$, we get a filtration of the logarithmic units $ \widetilde{ \mathcal{U}}_{K_{\mathfrak{p}}}$  as follows:

\begin{itemize}

\item[-]if $\mathfrak{p} \vert \ell$ by raising the local norm $N_{K_{\mathfrak{p}}/\mathbb{Q}_{\ell}}$ the filtration of
units on $\mathbb{Q}_{\ell}$ : $\widetilde{ \mathcal{U}}_{K_{\mathfrak{p}} }^{n}  =\{ x \in \mathcal{R}_{K_{\mathfrak{p}}} \mid N_{K_{\mathfrak{p}}/\mathbb{Q}_{\ell}}(x) \in (\ell)^{\ell^{n} \mathbb{Z}_{\ell}} \}.$                                                                                                                                                           
So, we have $\bigcap_{n} \widetilde{ \mathcal{U}}_{K_{\mathfrak{p}}} ^{n}=\{ x \in \mathcal{R}_{K_{\mathfrak{p}}} \mid N_{K_{\mathfrak{p}}/\mathbb{Q}_{\ell}}(x)=1 \}$. But by class field theory, the compositum $K_{\mathfrak{p}}
\mathbb{Q}_{\ell}^{ab}$ is fixed by the kernel of the norm. Thus the decreasing sequence  $(\widetilde{ \mathcal{U}}_{K_{\mathfrak{p}}} ^{n})_{n}$ is not exhaustive.

\item[-] if  $p \not \mid \ell $, as $\widetilde{ \mathcal{U}}_{K_{\mathfrak{p}} } \simeq \mu_{\mathfrak{p}} $
 the natural filtration is ${0} \subset \mu_{\mathfrak{p}}$.

\end{itemize}

\end{itemize}

\subsection{The logarithmic Artin map}

\noindent Let $K$ be a number field, the $\ell$-group of logarithmic divisors of $K$ is ~\cite[Def.2.1]{2}:
$$ \mathcal{D}\ell_{K}=\mathcal{J}_K / \widetilde{\mathcal{U}}_{K} \simeq  \bigoplus_{\mathfrak{p}}  \mathbb{Z}_{\ell} \;  \mathfrak{p}$$
through the logarithmic valuations $(\tilde v_{\mathfrak{p}})_{\mathfrak{p}}$ it can be identified to a free $\mathbb{Z}_{\ell}$-module built on finite primes of $K$.

\smallskip

\begin{defi}

Let $L/K$ be a finite and abelian $\ell$-extension, $\mathfrak{p}$ a prime of $K$ logarithmically unramified in $L$,
\noindent  $\mathcal{D}\ell_{K}$ be the group of logarithmic divisors of $K$ , $\tilde{\mathfrak{f}}_{L/K}$ be the global logarithmic conductor $L/K$ and $\mathcal{D}\ell_{K}^{\tilde{\mathfrak{f}}_{L/K}}$ the sub-group of logarithmic divisors prime to $\tilde{\mathfrak{f}}_{L/K}$.

\noindent{We define the logarithmic Frobenius of a prime $\mathfrak{p}$ logarithmically unramified, as follows: }

$$ (\widetilde{\frac{L/K}{\mathfrak{p}}})=([\tilde \pi_{\mathfrak{p}}], L/K)$$
where $\tilde \pi_{\mathfrak{p}}$ is the logarithmic uniformising element, defined in proposition 3.7.3 and
$[\tilde \pi_{\mathfrak{p}}]$ the image of $\tilde{\pi_{\mathfrak{p}}}$ in $\mathcal{J}_K.$

\noindent{We extend this map by multiplicativity: }

\begin{center}
$\begin{array}{ccccc}
\widetilde{(\frac{L/K}{.})}   &:  \mathcal{D}\ell_{K}^{\tilde{\mathfrak{f}}_{L/K}} &  & \to & \mathrm{Gal}(L/K) \\
& & \mathfrak{p} & \mapsto & (\widetilde{\frac{L/K}{\mathfrak{p}}})\\
\end{array}$
\end{center}

\end{defi}

\noindent{\textbf{Few remarks:}}
\begin{itemize}

\item[1)]The previous application is extended by multiplicativity to $\mathcal{D}\ell_{K}^{\tilde{\mathfrak{f}}_{L/K}}$  as by  hypothesis $L/K$  is a $\ell$-extension.

\item[2)]This map is a surjective $\mathbb{Z}_{\ell}$-morphism, due to the surjectivity of the global symbol. 

\item[3)]The motivation of this definition is the fact that in abstract class field theory ~\cite[Prop.3.4 ]{1}, if 
 $a \in A_K$, $(\;, L/K)$ the norm residue symbol of $L/K$ satisfies:
$$ (a, \tilde{K}/K)=\phi_{K}^{v_{K}(a)}$$
where $\phi_{K}$ is the generic Frobenius of $\tilde{K}/K$.

\item[4)] When $\mathfrak{p} \not \mid \ell$, the usual valuation and the logarithmic one are proportional: consequently $\tilde \pi_{\mathfrak{p}} =\pi_{\mathfrak{p}}$, the $\mathbb{Z}_{\ell}$-unramified extension and the $\mathbb{Z}_{\ell}$-cyclotomic one are the same. Thus the classical Artin symbol and the logarithmic one coincide on ideals.

\end{itemize}

\begin{proposition}{Properties of the logarithmic Artin map}

\noindent{Let $L/K$ be a finite and abelian $\ell$-extension,  $\tilde{\mathfrak{f}}_{L/K}$ the logarithmic global conductor. Then we have:}

\begin{itemize} 

\item[i)] If $M$ is between $L$ and $K$, the restriction of $\widetilde{(\frac{ L/K} {\mathfrak{a}} )} $  to $M$ is $\widetilde{(\frac{ M/K} {\mathfrak{a}} )} $, for all $ \mathfrak{a} \in \mathcal{D}\ell_{K}^{\tilde{\mathfrak{f}}_{L/K}} $
\medskip

\item[ii)] Let $L$ and $L'$ be abelian $\ell$-extensions, let's consider $\textrm{Gal}(LL'/K)$ as a sub-group of $\textrm{Gal}(L/K) \times  \textrm{Gal}(L'/K)$, through the map $\sigma \in \textrm{Gal}(LL'/K) \rightarrow ( \sigma|_{L}, \sigma|_{L'})
\in \textrm{Gal}(L/K) \times  \textrm{Gal}(L'/K)$, then $\widetilde{(\frac{ LL'/K} {\mathfrak{a}} )}$ is $ (\widetilde{ {(\frac{ L/K} {\mathfrak{a}} )}}, \widetilde{{(\frac{ L'/K} {\mathfrak{a}} )}})$  for all $\mathfrak{a} \in \mathcal{D}\ell_{LL'}^{\tilde{\mathfrak{f}}_{LL'/K}}$
\medskip

\item[iii)] Let $K'$ be any sub-extension of $K$, and let's consider $\textrm{Gal(LK'/K')}$ as isomorphic, by restriction to $K$,  to a sub-group of $\textrm{Gal}(L/K)$. The restriction of $\widetilde{{(\frac{ LK'/K'} {\mathfrak{a'}} )}}$ to $K$ is $\widetilde{{(\frac{K} {N_{K'/K}\mathfrak{a'}} )}}$ for all $\mathfrak{a'} \in \mathcal{D}\ell_{K'}^{\tilde{\mathfrak{f}}_{L/K'}}$
\medskip

\item[iv)] In particular, if $M$ is a field between $L$ and $K$, we have  $\widetilde{{(\frac{ L/M} {\mathfrak{U}} )}}=\widetilde{{(\frac{ L/K}{N_M\mathfrak{U}} )}}$, for all $\mathfrak{U} \in \mathcal{D}\ell_{M}^{\tilde{\mathfrak{f}}_{L/K}}$.

\end{itemize}

\end{proposition}

\begin{proof}

\noindent i) It sufficies to check the property on each prime $\mathfrak{p}$. By the fonctoriality property of the reciprocity map, we have: $\textrm{res}_{M} \circ \phi_{L/K}  (\tilde \pi_{\mathfrak{p}})=\phi_{M/K} (\tilde \pi_{\mathfrak{p}})$.

\noindent ii) is a consequence of i)

\noindent iii) We have to prove the equality between $\widetilde{{(\frac{ LK'/K'} {\mathfrak{a'}} )}}|_{K}=\phi_{LK'/K'}(\mathfrak{a'})$ and $\widetilde{{(\frac{ L} {N_{K'}\mathfrak{a'}} )}}= \phi_{L/K}(N_{K'}(\mathfrak{a'}))$
for all $\mathfrak{a'} \in \mathcal{D}\ell_{K'}^{\tilde{\mathfrak{f}}_{L/K'}}$ . By the fonctoriality  property of the reciprocity map,
we obtain this commutative diagramm:
 $$\begin{CD}
  \textrm{Gal}(L'/K')@>{r_{L'/K'}}>>\; \mathcal{J}_{K'} /N_{L'/K'}(\mathcal{J}_{L'}) \mathcal{R}_{K'}  \\
@VVV  @VV_{N_{K'/K}}V\\
 \textrm{Gal}(L/K)  @>{r_{L^{\;}/K^{\;}}} >>\mathcal{J}_K /N_{L/K}(\mathcal{J}_L) \mathcal{R}_{K}
\end{CD}$$

\noindent{where $r_{L/K}$ is the reciprocity map for global $\ell$-adic class field theory.} Consequently we have the following diagramm for the logarithmic global symbol: 
\medskip
 $$\begin{CD}
  \mathcal{J}_{K'}@>{(\;\cdot\;,L'/K') }>>\; \textrm{Gal}(LK'/K')  \\
@V_{N_{K'/K}}VV  @VV_{\textrm{res}}V\\
  \mathcal{J}_K@>{(\;\cdot\;,L^{\;}/K^{\;}) } >>\textrm{Gal}(L^{\;}/K^{\;}) 
\end{CD}$$

\noindent and we deduce the property from $\textrm{res} \circ (\;\cdot\;,LK'/K')=(\;\cdot\;,L/K) \circ N_{K'/K}$.

\noindent iv) is a particular case of iii) taking $K'=L$

\end{proof}

\begin{defi}
The kernel $\mathcal{A}\ell_{L/K}$ of the previous application $(\widetilde{\frac{L/K}{.\;}})$  is called the Artin logarithmic sub-module. For every modulus $\mathfrak{m}$ divisible by the global logarithmic conductor of  $L/K$, we put $\mathcal{A}\ell_{L/K,\mathfrak{m}}= \mathcal{A}_{L/K} \cap \mathcal{D}\ell_{K}^{\mathfrak{m}}$.
\end{defi}

\begin{defi}
Let $L/K$ be an abelian $\ell$-extension, $\mathfrak{m}$ a modulus dividing the logarithmic global conductor, we define:
$$ \mathcal{J}_{K}^{(\mathfrak{m})}=\prod_{\mathfrak{p} \not \mid \mathfrak{m} } \mathcal{R}_{K_{\mathfrak{p}}} \prod_{\mathfrak{p} \vert \mathfrak{m} } \widetilde{\mathcal{U}}_{K_{\mathfrak{p}}}^{v_{\mathfrak{p}}(\mathfrak{m})} $$
$$\mathcal{R}_{K}^{(\mathfrak{m})}= \mathcal{R}_{K} \cap \mathcal{J}_{K}^{(\mathfrak{m})}$$

\end{defi}

\begin{theorem}
Let $L/K$ be a finite and abelian $\ell$-extension, we have :
$$ \mathcal{A}\ell_{L/K}= \textrm{Gal}(L/K) \simeq  \mathcal{P}\ell_{K}^{(\tilde{\mathfrak{f}}_{L/K})} \cdot N_{L/K}( \mathcal{D}\ell_{L}^{\tilde{\mathfrak{f}}_{L/K}} )$$
where $\mathcal{P}\ell_{K}^{(\tilde{\mathfrak{f}}_{L/K})}$ is the sub-module of logarithmic principal divisors, image of the elements of $\mathcal{R}_{K}^{(\tilde{\mathfrak{f}}_{L/K})}$.
\end{theorem}

\begin{proof}
Let's denote here the logarithmic global conductor of $L/K$ by $\tilde{\mathfrak{f}}$.
By global $\ell$-adic class field theory, we know:
$$ \textrm{Gal}(L/K) \simeq \mathcal{J}_{K}/N_{L/K} ( \mathcal{J}_{L})  \mathcal{R}_{K}.$$

\noindent By the $\ell$-adic approximation lemma
{~\cite[II.2]{4}}, we know that the morphism of semi-localization
$ \mathcal{R}_{K} \longrightarrow \prod_{\mathfrak{p} \in S} \mathcal{R}_{K_{\mathfrak{p}}}$
 is surjective for all finite set of primes $S$. 

\noindent It follows:
$$\mathcal{J}_{K}= \mathcal{J}_{K}^{\widetilde{\mathfrak{f}}} \mathcal{R}_{K}.$$ 

\noindent Thus we get:
$$ \textrm{Gal}(L/K) \simeq \mathcal{J}_{K}^{\tilde{\mathfrak{f}}}  /N_{L/K} ( \mathcal{J}_L^{\tilde{\mathfrak{f}}} )  \mathcal{R}_{K}^{\tilde{\mathfrak{f}}}.$$
\noindent But the elements of $\widetilde{\mathcal{U}}_{K}^{(\tilde{\mathfrak{f}})}=\prod_{\mathfrak{p} \not \mid \tilde{\mathfrak{f}} } \widetilde{\mathcal{U}}_{K_{\mathfrak{p}}}  \prod_{\mathfrak{p} \vert \tilde{\mathfrak{f}} } \widetilde{\mathcal{U}}_{K_{\mathfrak{p}}}^{v_{\mathfrak{p}}(\tilde{\mathfrak{f}})} $  are norms : if $\mathfrak{p} \not \mid \tilde{\mathfrak{f}}$, $\mathfrak{p}$ is logarithmically unramified and units are norms, moreover if $\mathfrak{p} \vert \tilde{\mathfrak{f}}$, the definition of the conductor implies that those elements are norms. Due to this remark, we deduce:
$$ \textrm{Gal}(L/K) \simeq \mathcal{J}_{K}^{\tilde{\mathfrak{f}}}  /N_{L/K} ( \mathcal{J}_{L}^{\tilde{\mathfrak{f}}})  \;\widetilde{\mathcal{U}}_{K}^{(\widetilde{\mathfrak{f}})} \mathcal{R}_{K}^{\tilde{\mathfrak{f}}}.$$
\noindent But by the next lemma, we have: $\widetilde{\mathcal{U}}_{K}^{(\tilde{\mathfrak{f}})} \mathcal{R}_{K}^{\tilde{\mathfrak{f}}}=
\widetilde{\mathcal{U}}_{K} \mathcal{R}_{K}^{(\tilde{\mathfrak{f}})}$.
Finally we obtain:
$$ \textrm{Gal}(L/K) \simeq \mathcal{D}\ell_{K}^{\tilde{\mathfrak{f}}} / \mathcal{P}\ell_{K}^{(\tilde{\mathfrak{f})}} \cdot N_{L/K}( \mathcal{D}\ell_{L}^{\tilde{\mathfrak{f}} } ).$$

\end{proof}

\begin{lemma}
With the same notations, we get:
$$\widetilde{\mathcal{U}}_{K}^{(\tilde{\mathfrak{f}})} \mathcal{R}_{K}^{\tilde{\mathfrak{f}}}=
\widetilde{\mathcal{U}}_{K} \mathcal{R}_{K}^{(\tilde{\mathfrak{f}})}.$$

\end{lemma}
\smallskip

\begin{proof}
Let's take $\alpha \in \widetilde{\mathcal{U}}_{K}^{(\tilde{\mathfrak{f}})} \mathcal{R}_{K}^{\tilde{\mathfrak{f}}}$ and write
$\alpha =u r$ with $u \in \widetilde{\mathcal{U}}_{K}^{(\tilde{\mathfrak{f}})} $ and $r \in \mathcal{R}_{K}^{\tilde{\mathfrak{f}}}$ . This last condition implies that for all primes $\mathfrak{p}$ dividing the conductor, the local component
$r_{\mathfrak{p}}$ is a logarithmic unit. The approximation lemma gives a principal idele $\beta$,
whose local components for the primes dividing the conductor are $r_{\mathfrak{p}}$.
Finally we get $\alpha \beta^{-1} \in \widetilde{\mathcal{U}}_{K} \mathcal{R}_{K}^{(\tilde{\mathfrak{f}})} $ : this is the first inclusion. 

\noindent Conversely, let's consider $\alpha \in \widetilde{\mathcal{U}}_{K} \mathcal{R}_{K}^{(\tilde{\mathfrak{f}})}$ and
write $\alpha =u r$ with now $u \in \widetilde{\mathcal{U}}_{K}$ and $r \in \mathcal{R}_{K}^{(\tilde{\mathfrak{f}})}$.
As $u \in \widetilde{\mathcal{U}}_{K}$, the local component $u_{\mathfrak{p}}$ is a logarithmic unit and in particular  for $\mathfrak{p} \vert \tilde{\mathfrak{f}} $.  Due to the approximation lemma we get a principal idele $\beta$, whose local components for the primes $\mathfrak{p} \vert \tilde{\mathfrak{f}} $ are $u_{\mathfrak {p}}$. Finally $\alpha \beta^{-1} \in \widetilde{\mathcal{U}}_{K}^{(\tilde{\mathfrak{f}})} \mathcal{R}_{K}^{\tilde{\mathfrak{f}}}$. The equality follows.
\end{proof}

\noindent \textbf{Remark :}
Let's notice the analogy between the expression of the classical Takagi's group and the logarithmic Artin sub-module.

\begin{theorem}

Let $L/K$ be a finite and abelian $\ell$-extension, $\mathfrak{m}$ a modulus of $K$  divisible by $\tilde{\mathfrak{f}}_{L/K}$, then  $(\widetilde{\frac{L/K}{.\;}})$  restricted to $\mathcal{D}\ell_{K}^{\mathfrak{m}}$  is surjective and leads to an isomorphism between  $\mathcal{D}\ell_{K}^{\mathfrak{m}}/\mathcal{A}\ell_{L/K,\mathfrak{m}}$ and $\textrm{Gal}(L/K)$.
\end{theorem}

\subsection{Example: the quadratic case}

\noindent{To better understand the differences between the classical and the logarithmic case, let's focus on a quadratic extension: $\mathbb{Q}(\sqrt{d})/\mathbb{Q}$. This is a $\ell$-extension for $\ell=2$}, thus classical ramification and logarithmic ramification only differ for the primes above $\ell$ (cf \cite{2}). 

\noindent 1) Let's consider first a prime $p$ which is not above $2$ :

\begin{itemize}

\item either the prime $p$ is ramified in the classical and the logarithmic sense: the classical Frobenius and the logarithmic one are not defined. 

\item or the prime $p$ is unramified both in the classical and the logarithmic sense: we have to consider two cases

\begin{itemize}

\item either $p$ is inert and logarithmically inert: the decomposition sub-group is then isomorphic to the Galois group of the quadratic extension. It contains the identity and the usual Frobenius which coincides with the logarithmic one.

 \item or $p$ is completely splitten in the classical and the logarithmic sense : the decomposition sub-group
is trivial, and so are the classical Frobenius and the logarithmic one.
 
\end{itemize}

\end{itemize}
\noindent 2) Let's study the case of $2$.

\noindent{$2$ is logarithmically unramified, means by definition:}

\begin{center}
$[ \mathbb{Q}_{2}(\sqrt{d})  \cap \widehat{\mathbb{Q}_{2}^{c}} : \mathbb{Q}_{2}(\sqrt d) ] =1 \Leftrightarrow \mathbb{Q}_{2}(\sqrt{d}) \subseteq \mathbb{Q}_{2}^{c} $
\end{center}

\noindent{with $ \widehat{\mathbb{Q}_{2}^{c}}$ the  $\widehat{\mathbb{Z}}$-cyclotomic extension of $\mathbb{Q}_{2}$ and $\mathbb{Q}_{2}^{c}$ the $\mathbb{Z}_{2}$-cyclotomic extension of $\mathbb{Q}$ }. $ \widehat{\mathbb{Q}_{2}^{c}}$  is just the compositum of all  $\mathbb{Z}_{q}$-cyclotomic extensions for all primes $q$, those extensions are linearly separated and their Galois group is isomorphic to  $\mathbb{Z}_{p}$. The previous equivalence is due to the fact that 
$\mathbb{Z}_{2}$ is the only one which has a quotient isomorphic to $ \mathbb{Z}/ 2. \mathbb{Z} $.

\noindent{The $\mathbb{Z}_{2}$-cyclotomic extension of $\mathbb{Q}$ is cyclic, thus we may focus on the first step of the tower:}
$\mathbb{Q}_{2}(\sqrt{d}) \subseteq  \mathbb{Q}_{2}(\sqrt{2}).$ So we have two cases: 

$$ \mathbb{Q}_{2}(\sqrt{d})= \mathbb{Q}_{2} \qquad \textrm{or} \qquad \mathbb{Q}_{2}(\sqrt{d})=\mathbb{Q}_{2}(\sqrt{2}) $$

\noindent{thus}

$$ d \in {\mathbb{Q}_{2}^{\times}}^{2} \; \textrm{i.e.} \; d \equiv 1 \mod 8 \qquad \textrm{or} \qquad d/2 \in {\mathbb{Q}_{2}^{\times}}^{2} \; \textrm{i.e.} \; d/2 \equiv 1 \mod 8 $$ 

\noindent{as the squares of $\mathbb{Q}_{2}$ are of the shape $2^n.u$ where $u$ is a
2-adic unit congruent to $1 \mod 8$ and $n$ is a rational even number . }

\begin{itemize}

\item If $ d \equiv 1 \mod 8$,  $\mathbb{Q}_{2}(\sqrt{d})= \mathbb{Q}_{2} $, the local extension is trivial, the same is true for 
the decomposition subgroup. In this case, $2$ is completely splitten in the classical and in the logarithmic sense.

\item If $ d \equiv 2 \mod 16$, $2$ is ramified in the classical sense but logarithmically inert. Its decomposition sub-group contains two elements : the trivial one and the other one which is the logaritmic Frobenius.
$\mathbb{Q}_{2}(\sqrt{d})/ \mathbb{Q}_{2}$ is the maximal real sub-extension of $ \mathbb{Q}_{2}(\zeta_{8})/ \mathbb{Q}_{2}$, which is fixed by <$-1$>. Moreover  $\textrm{Gal}(\mathbb{Q}_{2}(\zeta_{8})/ \mathbb{Q}_{2}) \simeq ( \mathbb{Z}/8.\mathbb{Z})^{\times} $, thus we may explicit the action of the logarithmic Frobenius of $2$ :
$ \widetilde{(\frac {K/\mathbb{Q}} {2} )} ( \zeta_{8}) = \zeta_{8}^{3} . $ 
\end{itemize}

\subsection{Generalization: from quadratic to $\ell$-extensions}

\begin{proposition}
Let $K/\mathbb{Q}$ be an abelian $\ell$-extension with $\ell \ne 2$ and $p$ a prime of $\mathbb{Q}$.

\begin{itemize}

\item For $p \ne \ell$, classical and logarithmic ramification coincide.
If $p$ is unramified, the classical Frobenius and the logarithmic one are the same.

\item For $p=\ell$ we have several cases : ( let $\mathfrak{p}$ be a prime of $K$ above $\ell$)

\begin{itemize}

\item either $K_{\mathfrak{p}}=\mathbb{Q}_{\ell}$ : $\ell$ is then completely splitten in the classical and the logarithmic sense.
 The decomposition sub-group is trivial, classical and logarithmic Frobenius are 
 equal and both trivial.

\item and $\mathbb{Q}_{\ell} \ne K_{\mathfrak{p}} \subset \mathbb{Q}_{\ell}^{nr}$ : $\ell$ is classically unramified and  logarithmically ramified, only the classical Frobenius exists

\item or $\mathbb{Q}_{\ell} \ne K_{\mathfrak{p}} \subset \mathbb{Q}_{\ell}^{c}$ : $\ell$ is logarithmically unramified but classically ramified, only the logarithmic Frobenius exists

\item in any other case, $\ell$ is both classically and logarithmically ramified.

\end{itemize}

\end{itemize}

\end{proposition}

\begin{proof}
If $\ell \ne 2$, the maximal abelian pro-$\ell$-extension $\mathbb{Q}_{\ell}^{ab}$ of $\mathbb{Q}_{\ell}$ is the compositum of $\mathbb{Z}_{\ell}$-unramified extension $\mathbb{Q}_{\ell}^{nr}$ and the cyclotomic one $\mathbb{Q}_{\ell}^{c}$.

Let's assume $p \ne \ell$, then the classical and the logarithmic uniformizing elements are equal, the local symbol coincide :
thus the classical and the logarithmic Frobenius are the same.

Let's assume $p = \ell$ and that $\ell$ is logarithmically unramified in $K$. As we work with  $\ell$-extensions, we obtain:
$$ K_{\mathfrak{p}} \subseteq \mathbb{Q}_{\ell}^{c},$$
\noindent where $ \mathfrak{p}$ is a prime of $K$ above $\ell$.
As $K/\mathbb{Q}$ is a finite extension, we may concentrate on the first steps of the $\mathbb{Z}_{\ell}$-cyclotomic extension : there exists an integer $n$ such that 
$$K_{\mathfrak{p}} \subseteq B_{n+1}, $$
\noindent where $B_{n+1}$ is the maximal real subextension of $\mathbb{Q}_{\ell}(\zeta_{\ell^{n+1}})$ with $\zeta_{\ell^{n+1}}$ a $\ell^{n+1}$-root of unity and  $[B_{n+1} : \mathbb{Q}_{\ell}]= \ell^{n}$. This yields to:
$$ K_{\mathfrak{p}} =\mathbb{Q}_{\ell} \; \textrm{or} \;  K_{\mathfrak{p}}=B_{n+1}.$$
In the first case, the local extension is trivial: $\mathfrak{p}$ is completely splitten in the classical and in the logarithmic sense. In the second case, $\ell$ is ramified in the classical sense but  logarithmically unramified.

\end{proof}

\noindent{\sc Acknowledgement}\smallskip

\noindent{I would like to thank Jean-François Jaulent, my advisor, for our discusssions.}

\small{

}}

{\small

\noindent{Univ. Bordeaux,
IMB, UMR 5251,
351 Cours de la Libération,
F-33400 Talence, France}

\noindent{CNRS,
IMB, UMR 5251,
351 Cours de la Libération,
F-33400 Talence, France}
}
\smallskip

\noindent{\small{ Email address \smallskip
\tt stephanie.reglade@math.u-bordeaux1.fr
}}


\begin{thebibliography}{cc}

\bibitem[1] {1} {\sc J. Neukirch}, \emph{Class Field Theory}, Springer-Verlag, GTM 280, (1986)


\bibitem[2] {2}  {\sc J.-F. Jaulent},  \emph{Classes logarithmiques d'un corps de nombres},  J. Th\'eor. Nombres Bordeaux, \textbf{6}, (1994), 301--325. 

\bibitem [3] {3}  {\sc J.-F.  Jaulent},   \emph{Th\'eorie $\ell$-adique du corps des classes}, J. Th\'eor. Nombres Bordeaux, \textbf{10}, fasc.2 (1998),
355--397.

\bibitem[4] {4} {\sc G. Gras}, \emph{Class field theory from theory to practice}, Springer-Verlag, (2003)


\bibitem[5] {5} {\sc S. Reglade}, \emph{A formal approach "\`A  la Neukirch" of $\ell$-adic class field theory}, submitted

\bibitem[6] {6} {\sc  S. Reglade}, \emph{ Diff\'erentes approches de la th\'eorie $\ell$-adique du corps des classes}, Bordeaux
th\`ese, Th\'eor. Nombres (2014), 1--91

\bibitem[7] {7} {\sc  C. Brighi}, \emph{ Capitulation des classes logarithmiques et \'etude de certaines tours de corps de nombres},
th\`ese, Publ. Math. Fac. Sci. Metz, Th\'eor. Nombres (2007), 1--67



\end{thebibliography}
\end{document}